\documentclass[11pt,a4paper]{article}
\usepackage[utf8]{inputenc}
\usepackage[english]{babel}
\usepackage[T1]{fontenc}
\usepackage{amssymb,amsmath,amsthm,bm}
\usepackage{hyperref,cleveref,xspace}
\usepackage{graphicx,color}
\usepackage{comment,todonotes,cite}
\usepackage[babel=true]{csquotes}
\usepackage{systeme}
\usepackage{mathenv}
\usepackage[overload]{empheq}
\usepackage{tikz-qtree}
\usepackage{enumerate}
\newtheorem{theorem}{Theorem}
\newtheorem{proposition}[theorem]{Proposition}

\newtheorem{lemma}[theorem]{Lemma}

\newtheorem{observation}{Observation}

\usepackage[linesnumbered,ruled,vlined]{algorithm2e}

\usepackage[left=2cm,right=2cm,top=2cm,bottom=2cm]{geometry}

\title{Enumerating $k$-arc-connected orientations}

\author{Sarah Blind\thanks{Université de Lorraine, LGIPM, F-57000 Metz, France\newline
\indent Email: \texttt{sarah.blind@univ-lorraine.fr}} \and Kolja Knauer\thanks{Departament de Matem\`atiques i Inform\`atica,
Universitat de Barcelona (UB), Barcelona, Spain}~\thanks{Aix-Marseille Univ, Universit\'e de Toulon, CNRS, LIS, Marseille, France  \newline
\indent Email: \texttt{\{kolja.knauer,petru.valicov\}@lis-lab.fr}}
\and Petru Valicov\footnotemark[2]~\thanks{LIRMM, CNRS, Université de Montpellier, Montpellier, France.}}

\begin{document}
\maketitle


\begin{abstract}
We study the problem of enumerating the $k$-arc-connected orientations of a graph $G$, i.e., generating each exactly once. A first algorithm using submodular flow optimization is easy to state, but intricate to implement. In a second approach we present a simple algorithm with $O(knm^2)$ time delay and amortized time $O(m^2)$, which improves over the analysis of the submodular flow algorithm.
As ingredients, we obtain enumeration algorithms for the $\alpha$-orientations of a graph $G$ in $O(m^2)$ time delay and for the outdegree sequences attained by $k$-arc-connected orientations of $G$ in $O(knm^2)$ time delay.

\end{abstract}

\section{Introduction}
In an enumeration problem one receives a typically small input and wants to output every element of a typically large resulting set exactly once\footnote{We use the term \emph{enumeration} instead of the sometimes used terms \emph{generation} or \emph{listing}.}.
Since the runtime of an enumeration algorithm depends on the output, usually one measures how much it exceeds the size of the output in terms of the size of the input. More precisely, one wants to control the \emph{delay}, i.e., the maximum time between two consecutive outputs (including the time before the first and after the last output) in terms of the input or at least the average over these, called the \emph{amortized time}. Clearly, the delay is an upper bound for the amortized time.

An instance that illustrates this challenge perfectly is given an undirected (not necessarily simple) graph $G$, enumerate all its orientations with a given property. Many types of orientations have been studied with respect to their enumeration complexity, see~\cite{ConteThesis} for an overview. Here we give just some examples, before describing the set-up of this paper.

The set of all orientations of $G=(V,E)$ can be identified with the set of vectors $\{0,1\}^m$ where $m=|E|$. Thus, enumerating all orientations of $G$ using a Gray code can be done with constant delay, see~\cite{DBLP:journals/jacm/Ehrlich73}. 
It gets more interesting when enumerating \emph{acyclic} orientations, i.e. those that have no directed cycles. In~\cite{Squire98} an algorithm for enumerating all acyclic orientations with $O(n^3)$ time delay  but linear amortized time $O(n)$ was given, where $n=|V|$. In~\cite{Barbosa} the delay was reduced to $O(mn)$ with an increase in amortized time to $O(m+n)$. Another improvement was obtained in~\cite{Conte2018}, where an algorithm of delay and amortized time $O(m)$ is given. 

\emph{Strongly connected} orientations are those such that for any two vertices $u,v\in V$ there is a directed path from $u$ to $v$. In~\cite{Con-16} an enumeration algorithm of strongly connected orientations with $O(m)$ time delay is given.
The main objective of the present paper is to enumerate a parameterized version of strong orientations, generalizing the above. Namely, we enumerate the \emph{$k$-arc-connected orientations} of $G$, i.e., those where at least $k$ arcs have to be removed in order to destroy the strong connectivity\footnote{When there is no ambiguity, we will abbreviate $k$-arc-connected by saying \emph{$k$-connected}.}. Note that an orientation is strongly connected if and only if it is $1$-connected. See Figure~\ref{fig:orientations} for examples.

 \begin{figure}[h]
 \begin{center}
 
 \includegraphics[width=.9\textwidth]{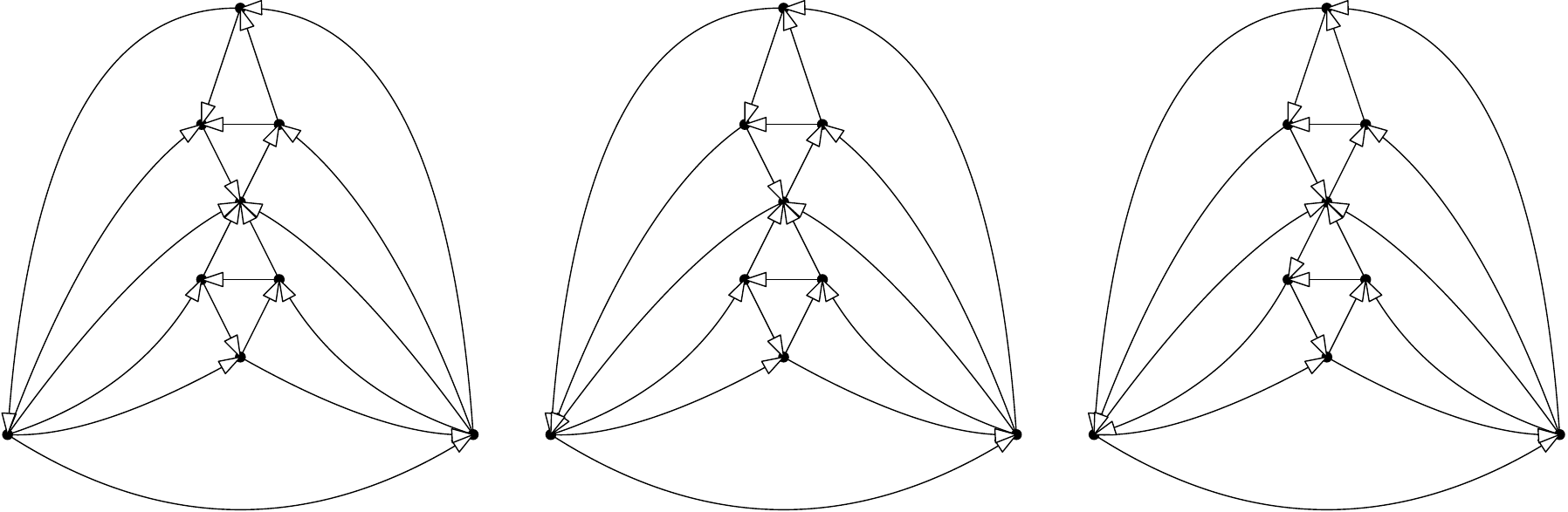}
 \caption{Three strongly connected orientations. Only those in the middle and on the right side are 2-arc-connected.} \label{fig:orientations}
\end{center}
 \end{figure}

The concept of $k$-connectivity is a classic in the theory of directed graphs.
Let us review some of the fundamental results of the area. A theorem of Nash-Williams~\cite{nash-williams_1960} asserts that $G$ admits a $k$-connected orientation if and only if $G$ is $2k$-edge connected. A result of Lovász~\cite{Lovasz:1979} yields an easy algorithm to find a $k$-connected orientation of $G$ if there is one, that runs in $O(n^6)$. More involved techniques have been developed to improve this runtime, see~\cite{Gabow94,LauY13}. Important contributions to the theory of $k$-connected orientations come from Frank. In particular, he showed that $k$-connected orientations are an instance of the theory of submodular flows, see~\cite{Frank82}. Finding a submodular flow (and hence a $k$-connected orientation) can be done in polynomial time. There is a considerable amount of literature proposing different algorithmic solutions for this task on simple graphs, multigraphs, and mixed graphs, see e.g.~\cite{Frank82,Gabow93,Gabow95,IwataK10}. The existence of these polynomial algorithms is the central fact behind a first enumeration algorithm that we propose here.
Another fundamental result of Frank is that any two $k$-connected orientations of $G$ can be transformed into each other by reversing directed paths and directed cycles~\cite{DBLP:journals/dm/Frank82}. Disassembling this result lies at the core of our second algorithm. 


All our algorithms are \textit{backtrack searches}, i.e., we build a tree whose vertices are partial solutions and the leaves are full solutions. We explore the tree on the fly by a depth-first search in order to reach all the leaves. The delay of the algorithm depends on the depth of the tree and the time spent at each node.

The first algorithm can be described as follows.
Using a min-cost $k$-connected orientation algorithm due to~\cite{IwataK10} and following ideas of~\cite{Bang-Jensen2018} we design an algorithm that decides in $O(n^3k^3+kn^2m)$ if a given mixed graph $G=(V,E\cup A)$ can be extended to an orientation that is $k$-connected.  This yields an enumeration algorithm for $k$-connected orientations with $O(m(n^3k^3+kn^2m))$ time delay.
The idea is to simply start with $G=(V,E)$ as the root vertex of the backtrack search tree and at each node create two children by picking an edge, fixing it to be oriented in one way or the other, and verifying if a $k$-connected extension exists. Since this tree has depth in $O(m)$ the previously stated runtime follows. Note that any other algorithm finding a $k$-connected extension of a partial orientation could be used in this approach, see for example~\cite{Frank82, Gabow93, Gabow95}. See \Cref{sec:submodular} for a detailed description.

The main part of the paper presents an alternative enumeration algorithm. Its main advantages are its simplicity and an improvement of the runtime. Moreover, it consists of two parts enumerating objects of independent interest. The idea is based on Frank's result~\cite{DBLP:journals/dm/Frank82} that any two $k$-connected orientations can be transformed into each other by reversals of directed cycles and directed paths. Our algorithm splits this into two parts: one reversing cycles, the other paths. 

By reversing directed cycles, we obtain an algorithm that enumerates all orientations of $G$ with a prescribed outdegree sequence also known as \emph{$\alpha$-orientations}. See for instance Figure~\ref{fig:orientations}, where the orientations on the middle and right have th same outdegrees and can be transformed into one another by revering a directed triangle. These orientations are of current interest with respect to computational properties (see e.g.~\cite{ACHKMSV19}) and model many combinatorial objects such as domino and lozenge tilings of a plane region \cite{Rem04,Thu90}, spanning trees of a planar graph~\cite{GL86}, perfect matchings (and $d$-factors) of a bipartite graph \cite{LZ03,P02,F04}, Schnyder woods of a planar triangulation \cite{Bre00}, Eulerian orientations of a graph \cite{F04}, and contact representations of planar graphs with homothetic triangles, rectangles, and $d$-gons \cite{Fel13,GLP12,FSS18b}. The set of $\alpha$-orientations of a planar graph can be endowed with a natural distributive lattice structure~\cite{F04} and therefore, in this case the enumeration can be done in linear amortized time~\cite{Habib01}. It is a famous question whether the enumeration of the elements of a distributive lattice can be done in constant amortized time~\cite{Pruesse1993}. This is open even when restricted to distributive lattices coming from the $\alpha$-orientations of a planar graph.
Our enumeration algorithm for general $\alpha$-orientations runs in $O(m^2)$ time delay  and is explained in~\Cref{sec:alpha}. 

Reversing directed paths, yields an enumeration algorithm for all outdegree sequences that are attained by $k$-connected orientations of $G$. This algorithm runs in $O(m^2kn)$ time delay and is explained in~\Cref{sec:outdegreesequences}.

In~\Cref{sec:together} we combine these algorithms, leading to an enumeration algorithm for $k$-connected orientations with time delay $O(m^2kn)$. The split into two algorithms leads to a structured traversal of the solution space which allows to prove an amortized runtime in $O(m^2)$. 

We close the paper in~\Cref{sec:discuss} with some open questions.

\section{Preliminaries}
In this section we introduce some digraph basics. All graphs we consider in this paper are loopless multigraphs and consequently their orientations also may have multiple parallel or anti-parallel arcs. We will also consider \emph{mixed multigraphs}. These are of the form  $G=(V,E\cup A)$, where $E$ is a multiset of undirected edges and $A$ is a multiset of directed arcs. Analogously to undirected graphs, an \emph{orientation} of a mixed graph consists in fixing a direction for each of its undirected edges.  

The digraph obtained from $D=(V,A)$ by reversing a set of arcs $B\subseteq A$ is denoted by $D^{B}$. If $A=B$ we write $D^-$ instead of $D^B$. Given an arc $a=(u,v)$, we denote by $a^-=(v,u)$ the reversed arc. Similarly, the reversed arc set of $B$ is denoted by $B^-=\{a^-\mid a\in B\}$.
Given a digraph $D=(V,A)$ and a subset $X\subseteq V$ we will denote its \emph{outdegree} $\delta_D^+(X)=|\{a=(u,v)\in A\mid u\in X\not\ni v\}|$. In the case where $X=\{v\}$ consists of a single element we just write $\delta_D^+(v)$.  

Note that by the definition of $k$-connectivity we have that $D$ is $k$-connected if and only if $\delta_D^+(X)\geq k$ for every $\emptyset\neq X\subsetneq V$. There is a directed version of Menger's Theorem that provides another characterization of $k$-connectivity, see~\cite{Menger}. For its statement, we denote by $\lambda(u,v)$ the maximum number of pairwise arc-disjoint directed paths from $u$ to $v$. 

\begin{theorem}[Local Menger Theorem for digraphs]\label{thm:Menger}
Let $D=(V,A)$ be a digraph and $u,v\in V$. We have $\lambda(u,v)=\min\{\delta^+_D(X)\mid u\in X\not\ni v\}$.
\end{theorem}

Thus, as a consequence of~\Cref{thm:Menger}, $D$ is $k$-connected if and only $\lambda(u,v)\geq k$ for all $u,v\in V$.

\section{A first enumeration algorithm}\label{sec:submodular}
We present a more detailed description of the first enumeration algorithm mentioned in the introduction. In~\cite{Bang-Jensen2018} the following result is shown based on a min-cost $k$-connected orientation algorithm due to~\cite{IwataK10}:
\begin{theorem}[\cite{IwataK10, Bang-Jensen2018}]\label{thm:feasibility}
It can be decided in time $O(k^3n^3+kn^2m)$, whether a mixed graph admits a $k$-connected orientation.
\end{theorem}
%

By Theorem~\ref{thm:feasibility}, one can easily design a backtrack search algorithm for enumerating all $k$-connected orientations with polynomial delay (see Algorithm~\ref{algo:submod}). At each step it takes a new edge and checks for both its orientations if there exists a $k$-connected orientation of the graph respecting the so far fixed orientation.

\bigskip

\begin{algorithm}[H]\label{algo:submod}
\caption{enumeration of $k$-connected orientations via min-cost orientations}
\KwIn{A graph $G = (V,E)$ and an integer $k\in\mathbb{N}$}
\KwOut{The $k$-connected orientations of $G$}
\BlankLine
\SetAlgoLined
\SetKwFunction{FEnumerate}{Enumerate}
\SetKwProg{Fn}{Function}{:}{end}
  \Begin{  \SetAlgoVlined
    Fix any linear ordering on $E$\;
    \FEnumerate{$G=(V,E,\emptyset)$}\;
  }
\bigskip

\Fn{\FEnumerate{$G=(V,E,F)$}}{
    \SetAlgoVlined
    \eIf{$E\neq \emptyset$}{
        Take the smallest $e = \{u,v\} \in E$\;
        \If{$G' := (V, E \setminus \{e\}, F \cup \{(u,v)\})$ admits a $k$-connected orientation}
            {\FEnumerate{$G'$}\;}
        \If{$G' := (V, E \setminus \{e\}, F \cup \{(v,u)\})$  admits a $k$-connected orientation}
            {\FEnumerate{$G'$}\;}
    }
    {Output $G$\;}
}
\end{algorithm}
\medskip

This algorithm takes as input an undirected graph $G=(V,E)$. Clearly, all $k$-connected orientations are produced and since each node built by this algorithm gives rise to disjoint branches it does not repeat the same solution twice.
The depth of the binary execution tree is $m$ and at each node we check the orientability of a mixed graph which is solvable in $O(k^3n^3+kn^2m)$ by~\Cref{thm:feasibility}. We conclude 

\begin{proposition}
Let $G$ be a graph and $k\in\mathbb{N}$. Algorithm~\ref{algo:submod} enumerates all $k$-connected orientations of $G$ with delay and amortized time in $O(m(k^3n^3+kn^2m))$.
\end{proposition}

A downside of Algorithm~\ref{algo:submod} is that~\Cref{thm:feasibility} is based on a submodular flow algorithm and its implementation is intricate. Moreover, the literature on submodular flows and the available runtime analyses in different settings such as simple and multigraphs are hard to extract. A more detailed discussion of these approaches is discussed in~\cite{BlindThesis}.

In the following we will give two simple algorithms of independent interest that combined yield an alternative solution for the enumeration of $k$-connected orientations. The resulting algorithm improves over the delay and amortized runtime of Algorithm~\ref{algo:submod}. 

\section{Orientations with prescribed outdegree sequence}\label{sec:alpha}

Let $G=(V,E)$ be a graph and $\alpha:V\to \mathbb{N}$. We say that an orientation $D$ of $G$ is an \emph{$\alpha$-orientation} if $\delta_{D}^{+}(v) = \alpha(v)$ for all $v\in V$. We will denote by $\mathcal{O}_{\alpha}(G)$ the set of all $\alpha$-orientations of $G$. In the present section we discuss the enumeration of the set $\mathcal{O}_{\alpha}(G)$.
Before presenting the algorithm let us give a folklore result about $\alpha$-orientations:
\begin{lemma}\label{lem:deg-id}
Let $G$ be a graph and $D$ and $D'$ be two orientations of $G$. We have $\delta_{D}^{+} = \delta_{D'}^{+}$ if and only if orientation $D'$ can be obtained from $D$ by reversing a set of arc-disjoint directed cycles.
\end{lemma}

\begin{proof}
Reversing the direction of a directed cycle does not change the outdegree function. Therefore, if $D'$ is obtained from $D$ by reversing a set of arc-disjoint directed cycles, then $\delta_{D}^{+} = \delta_{D'}^{+}$.

Let conversely $D=(V,A)$ and $D'=(V,A')$ be orientations of $G$ such that $\delta_{D}^{+} = \delta_{D'}^{+}$.
Consider the arcs $A\setminus A'$ of $D$ whose direction differ in $D'$. Observe that in the directed subgraph $D\setminus A'$ formed by these arcs, the indegree and the outdegree coincide at each node. Thus, $D\setminus A'$ is Eulerian and can be decomposed into an arc-disjoint union of directed cycles. This concludes the proof.
\qed
\end{proof}

Our enumeration algorithm for $\mathcal{O}_{\alpha}(G)$ takes a digraph $D=(V,A)\in\mathcal{O}_{\alpha}(G)$ and a set of \emph{fixed} arcs $F\subseteq A$ as input. Initially $D$ is an arbitrary $\alpha$-orientation and $F=\emptyset$. The algorithm recursively constructs all orientations of $\mathcal{O}_{\alpha}(G)$ such that at each call all possible $\alpha$-orientations  contain the current $F$.
At each recursive call, an arc $a\notin F$ will be fixed with the two possible cases to consider: The branch with $a=(u,v)$ is known to be non-empty since $D$ is a possible extension, thus we recurse with $D$ and $F\cup a$.
For the branch with $a^-=(v,u)$ we check if there is a directed path $P$ in $D\setminus F$ from $v$ to $u$. If this is the case, then the directed cycle $(P,a)$ is reversed and we recurse with the new $\alpha$-orientation $D^{P\cup\{a\}}$ and $F\cup \{a^-\}$.

The set of fixed arcs is extended along an arbitrary linear ordering of the arcs. Whenever all arcs are fixed we output the current orientation.
We give a presentation of our algorithm in pseudo-code (see Algorithm~\ref{algo:binary-tree}). 

\medskip
\begin{algorithm}[H]\label{algo:binary-tree}
\caption{Backtrack search for \emph{$\alpha$-orientations}}
\KwIn{A graph $G = (V,E)$ and $\alpha: V\to \mathbb{N}$}
\KwOut{All elements of $\mathcal{O}_{\alpha}(G)$}
\BlankLine

\SetAlgoLined
\SetKwFunction{FEnOPODS}{EnOPODS}
\SetKwProg{Fn}{Function}{:}{end}

\Begin{  \SetAlgoVlined
   \If{there exists $D=(V,A)\in \mathcal{O}_{\alpha}(G)$}{
   Fix an arbitrary linear order on $A$\;
   \FEnOPODS{$D$, $\emptyset$}\;
   }
   }

\bigskip

\SetKwFunction{FEnOPODS}{EnOPODS}
\SetKwProg{Fn}{Function}{:}{end}
\Fn{\FEnOPODS{$D$, $F$}}{
\SetAlgoVlined
\eIf{$F\neq A$}{
Take the smallest $a = (u,v)\in A\setminus F$\;
\FEnOPODS{$D$, $F\cup\{a\}$}\;
\If{$D\setminus F$ has a directed path $P$ from $v$ to $u$}{\FEnOPODS{$D^{P\cup\{a\}}, F\cup\{a^-\}$}\;}}
{Output $D$\;}
}
\end{algorithm}
\medskip


Algorithm~\ref{algo:binary-tree} takes an undirected graph $G = (V,E)$ as input, but after finding one initial $\alpha$-orientation $D$, it just consists of recursive calls of the function \FEnOPODS{$D$, $F$}. 

\begin{lemma}\label{thm:alpha}
Let $D\in\mathcal{O}_{\alpha}(G)$ and  $F\subseteq A$. The function \FEnOPODS{$D$, $F$} generates each orientation in $\mathcal{O}_{\alpha}(G)$ that coincides with $D$ on $F$ exactly once and runs with $O(m^2)$ time delay.
\end{lemma}
\begin{proof}
We prove that each of the claimed $\alpha$-orientations is generated exactly once by induction on $|A\setminus F|$. If $|A \setminus F| = 0$, then \FEnOPODS{$D$, $F$} = \FEnOPODS{$D$, $A$} $= \{D\}$ and we are done. Let now $|A \setminus F| > 0$ and $a = (u,v) \in A \setminus F$. By induction hypothesis \FEnOPODS($D,F\cup \{a\}$) generates each $\alpha$-orientation that coincides with $D$ on $F\cup \{a\}$ exactly once and \FEnOPODS{$D^{P\cup\{a\}}, F\cup\{a^-\}$} generates each $\alpha$-orientation that coincides with $D^{P\cup\{a\}}$ on $F\cup\{a^-\}$  exactly once. Clearly, both sets are disjoint since they differ with respect to the orientation of $a$. Thus, no repetitions are produced. Since $(P,a)$ is a directed cycle by~\Cref{lem:deg-id} we have $D^{P\cup\{a\}}\in\mathcal{O}_{\alpha}(G)$. Since $P\cap F=\emptyset$ we have that $D^{P\cup\{a\}}$ coincides with $D$ on $F$.

To see that we do not miss any orientation, we prove that if there is no directed path $P$ from $u$ to $v$ in $D\setminus F$, then there exists no $\alpha$-orientation fixing $F$ and reversing $a$. By contraposition, suppose that $D'=(V,A')$ is an $\alpha$-orientation that coincides with $D$ on $F$ but differs on $a$. Then by~\Cref{lem:deg-id}, there is a set of arc-disjoint directed cycles in $D'$ whose union is $A'\setminus A$. Since both digraphs coincide on $F$, these cycles are disjoint from $F$. Since both digraphs differ on $a$, one of the directed cycles $C$ contains $a^{-}$ in $D'$. Thus, the path $P=(C\setminus\{a^-\})^-$ is a directed path in $D\setminus F$ from $v$ to $u$.

For the complexity, note that in each recursion step the algorithm checks the presence of a directed path, which can be done by a single BFS from the source vertex $u$ towards the target $v$. In general the complexity of a BFS algorithm is $O(m+n)$, however in our case the BFS tree will be constructed only on the strongly connected component of $D$ containing $u$ and thus the complexity is in $O(m)$. The depth of our recursion tree is bounded by $m$. Thus, the total time delay is bounded by $O(m^2)$.
\qed
\end{proof}

Observe that Algorithm~\ref{algo:binary-tree} has to use a separate method for finding a first element $D\in\mathcal{O}_{\alpha}(G)$. It is well-known that this problem can be reduced to a flow-problem, see e.g.~\cite{F04}. Therefore, this preprocessing step can be done in $O(mn)$ time, see~\cite{Orlin13}. With~\Cref{thm:alpha} we obtain:

\begin{theorem}\label{thm:enumerateOalpha}
Let $G$ be a graph and $\alpha: V\to \mathbb{N}$. Algorithm~\ref{algo:binary-tree} enumerates $O_{\alpha}(G)$ with time delay in $O(m^2)$.
\end{theorem}

\section{Outdegree sequences}\label{sec:outdegreesequences}
In this section we will present an algorithm to enumerate the possible outdegree sequences among the $k$-connected orientations of a graph $G$. 

An easy consequence of~\Cref{lem:deg-id} is the following that can also be found in~\cite{DBLP:journals/dm/Frank82}.
\begin{lemma}\label{lem:alphakconnected}
If $D,D'\in\mathcal{O}_{\alpha}(G)$, then $D$ is $k$-connected if and only if $D'$ is $k$-connected.
\end{lemma}
\begin{proof}
Let $D,D'\in\mathcal{O}_{\alpha}(G)$ and $D$ be $k$-connected. Thus, for all $\emptyset\neq X \subsetneq V$ we have $\delta_{D}^{+}(X) \geq k$.
By \Cref{lem:deg-id} since $D'\in\mathcal{O}_{\alpha}(G)$ it can be obtained from $D$ by reversing a set of disjoint directed cycles. But reversing a directed cycle in a digraph does not change the outdegree of the subsets of vertices of $G$. Therefore, after reversing the set of directed cycles to obtain $D'$, we have $\delta_{D'}^{+}(X) \geq k$ for all $\emptyset\neq X \subsetneq V$ and $D'$ is $k$-connected.
\qed
\end{proof}

Given $G=(V,E)$, \Cref{lem:alphakconnected} allows to define a function  $\alpha: V\to\mathbb{N}$ to be \emph{$k$-connected} if there is some $k$-connected $D\in\mathcal{O}_{\alpha}(G)$. 
  In this case we call $\alpha$ a $k$-connected outdegree sequence. Having in mind that we want to enumerate all $k$-connected orientations of $G$ and already are able to enumerate $\mathcal{O}_{\alpha}(G)$ for any given $\alpha$, we are left with enumerating the $k$-connected  outdegree sequences.
%
%
Here in order to change the outdegree sequence of $D$ we will reverse a directed path $P_{uv}$ from $u$ to $v$ and thus increase $\delta_D^+(u)$ by one, decrease $\delta_D^+(v)$ by one, and leave the remaining outdegrees unchanged. More generally we have:

\begin{observation}\label{claim:break-connect}
Let $D$ be an orientation of graph $G$ and $X \subseteq V(G)$. If $D'$ is obtained from $D$ by reversing a path from a vertex $u$ to a vertex $v$, then we have 
\begin{enumerate}
	\item $\delta^+_{D'}(X) = \delta^+_{D}(X)$ if $u,v \in X$ or $u,v \notin X$
	\item $\delta^+_{D'}(X) = \delta^+_{D}(X) +1$ if $u \notin X$ and $v \in X$
	\item $\delta^+_{D'}(X) = \delta^+_{D}(X) -1$ if $u \in X$ and $v \notin X$
\end{enumerate}
\end{observation}

In order to maintain connectivity, we have to reverse paths without decreasing the number of arc-disjoint directed paths between pairs of vertices too much.

\begin{lemma}\label{lem:pathflipping}
Let $P_{uv}$ be a directed path from vertex $u$ to vertex $v$ in $D$. For all vertices $u',v'$, we have $\lambda_{D^{P_{uv}}}(u',v') \geq \min(\lambda_{D}(u,v)-1,\lambda_D(u',v'))$. Furthermore, we have $\lambda_{D^{P_{uv}}}(u,v)=\lambda_{D}(u,v)-1$.
\end{lemma}
\begin{proof}
With Menger's Theorem (\Cref{thm:Menger}) and Observation~\ref{claim:break-connect}, \Cref{lem:pathflipping} can be easily proved:

\smallskip
\begin{tabular}{rlr}
$\lambda_{D^{P_{uv}}}(u',v')$ &= $\min\{\delta^+_{D^{P_{uv}}}(X) \; | \; X \subset V, v' \notin X\ni u'\}$ &(\Cref{thm:Menger}) \\
~&$= \min(\{\delta^+_D(X) - 1 \; | \; X \subseteq V, v,v' \notin X\ni u',u\}$&(Observation~\ref{claim:break-connect})\\
~&$\cup\{\delta^+_D(X) \; | \; X \subseteq V, v' \notin X\ni u',u, v \text{ or } v',u, v \notin X\ni u'\}$ &\\
~&$\cup\{\delta^+_D(X) + 1 \; | \; X \subseteq V, u,v' \notin X\ni u',v\})$&~\\
~&$\geq \min\{ \lambda_D(u,v)-1, \lambda_D(u',v') \}$ &(\Cref{thm:Menger})
\end{tabular}
\smallskip

Note that in the case $u'=u$ and $v'=v$ the second and third part of the union in the above equation are empty. We thus get:

\smallskip
\begin{tabular}{rlr}
$\lambda_{D^{P_{uv}}}(u,v)$ &= $\min\{\delta^+_{D^{P_{uv}}}(X) \; | \; X \subseteq V, u' \in X, v' \notin X\}$ &(\Cref{thm:Menger}) \\
~&$=\min\{\delta^+_{D^{P}}(X)-1 \; | \; X \subseteq V, u' \in X, v' \notin X\}$ &(Observation \ref{claim:break-connect})\\
~&$=\lambda_D(u,v)-1$ &(\Cref{thm:Menger})
\end{tabular}

\qed
\end{proof}

 In a $k$-connected digraph $D$ we call a directed path $P$ \emph{flippable} if $D^P$ is $k$-connected. \Cref{lem:pathflipping} implies that a path from $u$ to $v$ is flippable, if and only if all of them are which is equivalent to $\lambda_D(u,v)>k$. In this case we call the pair $(u,v)$ \emph{flippable}. The following gives an algorithm, that is basically equivalent to the Edmonds-Karp algorithm for maximum flows~\cite{D70,EK72}:

\begin{lemma}\label{lem:computelambda}
Let $D$ be $k$-connected, it can be decided in time $O(km)$ if $(u,v)$ is flippable.
\end{lemma}
\begin{proof}
By~\Cref{lem:pathflipping} a simple algorithm consists in finding a directed path $P_{uv}$ in $D$, reverse it and iterate in $D^{P_{uv}}$. We have $\lambda_D(u,v)> k$ if and only if this procedure can be applied $k+1$ times. Each execution is a BFS, which yields the claimed runtime.
\qed
\end{proof}

The following is a slight refinement of a result of Frank~\cite{DBLP:journals/dm/Frank82} and constitutes the last ingredient for our algorithm:
\begin{lemma}\label{lem:deg-diff}
Let $G=(V,E)$ be a graph and $D$, $D'$ be two $k$-connected orientations of $G$.
For every $v\in V$ with $\delta_{D}^{+}(v) < \delta_{D'}^{+}(v)$, there exists $u\in V$ such that $\delta_{D}^{+}(u) > \delta_{D'}^{+}(u)$ and $(u,v)$ is flippable in $D$.
\end{lemma}

\begin{proof}
Let $v\in V$ be with $\delta_{D}^{+}(v) < \delta_{D'}^{+}(v)$. We will prove that there is a $u\in V$ such that $\delta_{D}^{+}(u) > \delta_{D'}^{+}(u)$ and $\delta_{D}^{+}(X) > k$ whenever $u \in X$ and $v \notin X$. By Menger's Theorem (\Cref{thm:Menger}) this implies that $(u,v)$ is flippable.
Since $\delta_{D}^{+}(v) < \delta_{D'}^{+}(v)$ and because $\sum\limits_{w\in V}\delta^+_D(w)=|E|=\sum\limits_{w\in V}\delta^+_{D'}(w)$, there exists at least one vertex $u$ such that $\delta_{D}^{+}(u) > \delta_{D'}^{+}(u)$ .
By contradiction suppose that every vertex $u$ with $\delta_{D}^{+}(u) > \delta_{D'}^{+}(u)$, is contained in a set $X$ with $\delta_D^+(X)=k$ and $v\notin X$.
Among all these sets for all such $u$, we consider those that are inclusion maximal and collect them in $\mathcal{X}$. 

Let us first prove that the members of $\mathcal{X}$ are pairwise disjoint. Indeed, consider two sets $X,X'$ with $\delta^+_D(X)=\delta^+_D(X')=k$ and $v\notin X \cup X'$ with $X\cap X'\neq\emptyset$. Since $X \cup X'$, $X \cap X'$ are both not empty, $X \cup X'\neq V$ and $D$ is $k$-connected, we know that $\delta_{D}^{+}(X \cup X') \geq k$ and $\delta_{D}^{+}(X \cap X') \geq k$. Therefore, we have
$$ k+k = \delta_{D}^{+}(X) + \delta_{D}^{+}(X') \geq \delta_{D}^{+}(X \cup X') + \delta_{D}^{+}(X \cap X') \geq k+k$$
This gives $\delta_{D}^{+}(X \cup X') = k$, i.e., $X$ or $X'$ is not maximal.

Now, let us count the number $c$ of edges of $G$ not contained in any subgraph of $G$ induced by $X\in\mathcal{X}$ and denote $Y=V\setminus\bigcup_{X\in\mathcal{X}}X$. Observe that since the elements of $\mathcal{X}$ are disjoint, in any orientation of $G$ the number $c$ is just the sum of outdegrees of $X\in\mathcal{X}$ plus the sum of outdegrees of vertices of $Y$. Thus, if we furthermore denote $t=|\mathcal{X}|$ we can do this counting with respect to $D$ and $D'$ and obtain the following contradiction

$$c =  \sum \limits_{X \in \mathcal{X}}\delta_{D}^{+}(X) + \sum \limits_{w \in Y}\delta_{D}^{+}(w) = k t + \sum \limits_{w \in Y}\delta_{D}^{+}(w) < \sum \limits_{X \in \mathcal{X}}\delta_{D'}^{+}(X) + \sum \limits_{w \in Y}\delta_{D'}^{+}(w)=c$$

More precisely, since $D'$ is $k$-connected we have that $kt\leq \sum \limits_{X \in \mathcal{X}}\delta_{D'}^{+}(X)$. Recall that we supposed that if $w$ is such that $\delta_{D}^{+}(w) > \delta_{D'}^{+}(w)$, then $w \notin Y$. Therefore, we have $\delta_{D}^{+}(w) \leq \delta_{D'}^{+}(w)$ for all $w \in Y$. Adding this to the fact that $\delta_{D}^{+}(v) < \delta_{D'}^{+}(v)$ and $v\in Y$, we obtain the strict inequality claimed above. This concludes the proof.
\qed
\end{proof}

Note that a complete analogue of~\Cref{lem:deg-diff} holds for the case $\delta_{D}^{+}(v) > \delta_{D'}^{+}(v)$:

\begin{lemma}\label{lem:deg-diff-bis}
Let $G=(V,E)$ be a graph and $D$, $D'$ be $k$-connected orientations of $G$.
For every $v\in V$ with $\delta_{D}^{+}(v) > \delta_{D'}^{+}(v)$, there exists $u\in V$ such that $\delta_{D}^{+}(u) < \delta_{D'}^{+}(u)$ and $(v,u)$ is flippable in $D$.
\end{lemma}

\begin{proof}
Observe that an equivalent statement of the lemma is that for every vertex $v$ such that $\delta_{D}^{-}(v) < \delta_{D'}^{-}(v)$, there exists a vertex $u$ such that $\delta_{D}^{-}(u) > \delta_{D'}^{-}(u)$ and $(v,u)$ is flippable in $D$. By fully reorienting all the arcs of $D$ and $D'$, we obtain graphs $D^-$ and $D'^-$ and get a further equivalent statement: for every vertex $v$ such that $\delta_{D^-}^{+}(v) < \delta_{D'^-}^{+}(v)$, there exists a vertex $u$ such that $\delta_{D^-}^{+}(u) > \delta_{D'^-}^{+}(u)$ and $(u,v)$ is flippable in $D^-$. This is precisely the statement of \Cref{lem:deg-diff}, so we are done.
\qed
\end{proof}

Now we are ready de describe the general enumeration algorithm for $k$-connected outdegree sequences:

\begin{algorithm}[H]\label{algo:tree}
\caption{Enumeration of $k$-connected outdegree sequences}
\BlankLine
\KwIn{A graph $G=(V,E)$, an integer $k$}
\KwOut{All $k$-connected outdegree sequences of $G$}
\BlankLine
\SetAlgoLined

\SetKwFunction{FEnODS}{EnODS}
\SetKwFunction{FReverseMinus}{$\text{Reverse}^-$}
\SetKwFunction{FReversePlus}{$\text{Reverse}^+$}
\SetKwProg{Fn}{Function}{:}{end}
  \Begin{  \SetAlgoVlined
    \If{there exists a $k$-connected orientation $D$ of $G$}{
            Fix an arbitrary linear order on $V$\;
            \FEnODS{$D$, $\emptyset$}\;
    }
  }

\bigskip

\Fn{\FEnODS{$D$, $F$}}{
    \SetAlgoVlined
    \eIf{$F \neq V$}{
        take the smallest $v\in V\setminus F$\;
        \FReverseMinus{$D$, $F$, $v$}\;
        \FReversePlus{$D$, $F$, $v$}\;
        \FEnODS{$D,F\cup\{v\}$}\;
    }
    {Output $\delta^+_D$\;}
  }

\bigskip

\SetKwFunction{FEnODS}{EnODS}
\SetKwFunction{FReverseMinus}{$\text{Reverse}^-$}
\SetKwFunction{FReversePlus}{$\text{Reverse}^+$}
\SetKwProg{Fn}{Function}{:}{end}

\Fn{\FReverseMinus{$D$, $F$, $v$}}{
\SetAlgoVlined
\If{there exists $u \in V\setminus F$ such that $(v,u)$ is flippable}{
Take a directed path $P_{vu}$ from $v$ to $u$ \;
\FReverseMinus{$D^{P_{vu}}$, $F$, $v$} \;
\FEnODS{$D^{P_{vu}}$, $F\cup\{v\}$}\;
}
}
\bigskip

\Fn{\FReversePlus{$D$, $F$, $v$}}
{  \SetAlgoVlined
\If{there exists $u \in V\setminus F$ such that $(u,v)$ is flippable}{
Take a directed path $P_{uv}$ from $u$ to $v$ \;
\FReversePlus{$D^{P_{uv}}, F, v$} \;
\FEnODS{$D^{P_{uv}}, F\cup\{v\}$}\;
}
}
\end{algorithm}

\begin{lemma}\label{thm:sequences}
Let $D$ be a $k$-connected orientation of $G=(V,E)$ and  $F\subseteq V$. The function \FEnODS{$D$,$F$} enumerates the $k$-connected outdegree sequences coinciding with $D$ on $F$ with time delay in $O(knm^2)$.
\end{lemma}
\begin{proof}
We show the first part of the lemma by induction on $|V\setminus F|$. If $|V\setminus F|=0$, then  \FEnODS{$D$,$F$} outputs $\delta^+_D$ and the claim holds. 
Consider now the case $|V\setminus F| > 0$ and let $v\in V\setminus F$ be the next vertex. By induction \FEnODS{$D$,$F\cup\{v\}$} generates every $k$-connected outdegree sequence coinciding with $D$ on $F\cup\{v\}$ exactly once. We have to show that \FReversePlus{$D$,$F$,$v$} (resp. \FReverseMinus{$D$,$F$,$v$})  enumerates all $k$-connected outdegree sequences coinciding with $D$ on $F$ and having outdegree of $v$ larger (resp. smaller) than $\delta^+_D(v)$. Also, we have to show that each of these outdegree sequences will be generated exactly once. Note that this implies that globally each solution is produced exactly once.

Let us prove this for \FReversePlus{$D$,$F$,$v$}. So let $D'$ be a $k$-connected orientation of $G$ such that $\delta_D^+(F)\equiv \delta_{D'}^+(F)$ and $\delta_D^+(v)<\delta_{D'}^+(v)$. By~\Cref{lem:deg-diff} there exists a vertex $u\in V\setminus F$ such that $(u,v)$ is flippable, i.e., for any path $P_{uv}$ the orientation $D^{P_{uv}}$ is $k$-connected, its outdegree sequence coincides with $D$ on $F$ and $\delta_D^+(v)+1=\delta_{D^{P_{uv}}}^+(v)$. 

We proceed by induction on $\delta_{D'}^+(v)-\delta_D^+(v)$ to show that $\delta^+(D')$ is enumerated exactly once. 
So for the base case $\delta_{D'}^+(v)-\delta_D^+(v)=1$ we have $\delta_{D'}^+(v)=\delta_{D^{P_{uv}}}^+(v)$ and by induction $\delta^+_{D'}$ will be enumerated exactly once by the next call of \FEnODS{$D^{P_{uv}}$,$F\cup\{v\}$} and not at all by \FReversePlus{$D^{P_{uv}}$, $F$, $v$} since the latter outputs degree sequences with $\alpha(v)>\delta_{D^{P_{uv}}}^+(v)$. 
Suppose now that $\delta_{D'}^+(v)-\delta_D^+(v)>1$. We have $\delta_{D'}^+(v)-\delta_{D^{P_{uv}}}^+(v)<\delta_{D'}^+(v)-\delta_D^+(v)$, so by induction hypothesis \FReversePlus{$D^{P_{uv}}$,$F$,$v$} enumerates the outdegree sequence of $\delta_{D'}^+(v)$ exactly once.

The analogue proof works for \FReverseMinus using~\Cref{lem:deg-diff-bis}.

For the analysis of complexity note that in each call of \FReversePlus or \FReverseMinus for at most $n$ times it has to be checked if a pair $(u,v)$ is flippable. The latter can be done in time $O(km)$ by~\Cref{lem:computelambda}, finding a directed path from $u$ to $v$ is done in $O(m)$. So a call costs $O(knm)$. 

Finally, the depth of the recursion tree is in $O(m)$. To see this compare the $\delta^+_{D'}$ of a leaf orientation with the $\delta^+_{D}$ of orientation $D$ at the root. Between any two calls of \FEnODS, there will be a sequence of at most $\deg(v)$ calls of \FReversePlus or \FReverseMinus. This way $\delta^+_{D}$ will be approached to $\delta^+_{D'}$ coordinate by coordinate, where previous coordinates are not affected by modifications on latter coordinates. Thus, there are at most $\sum_{v\in V}\deg(v)=2m$ calls and we get an overall time delay of $O(knm^2)$.
\qed
\end{proof}

Note that Algorithm~\ref{algo:tree} has to use a separate method for finding a first $k$-connected orientation $D$ of $G$. This preprocessing step can be done in $O(k^3n^3+kn^2m)$~\cite{IwataK10}. On the other hand, recall that in a $k$-connected orientation we have $kn \leq m$. Therefore, together with~\Cref{thm:sequences} we obtain:
\begin{theorem}\label{thm:enumeratesequences}
Let $G$ be a graph and $k\in\mathbb{N}$. Algorithm~\ref{algo:tree} enumerates all $k$-connected outdegree sequences of $G$ in $O(knm^2)$ time delay.
\end{theorem}

\section{$k$-connected orientations}\label{sec:together}

Putting the above together we obtain an algorithm to enumerate $k$-connected orientations.

\begin{algorithm}[H] \label{algo:final}
    \caption{Simple enumeration of $k$-connected orientations}
    \LinesNumbered
    \BlankLine
    \KwIn{A graph $G=(V,E)$, an integer $k$}
    \KwOut{All $k$-connected orientations of $G$}
    \BlankLine
    \SetAlgoLined
    
\SetKwFunction{FEnODSPrime}{EnODS'}
\SetKwFunction{FReverseMinus}{$\text{Reverse}^-$}
\SetKwFunction{FReversePlus}{$\text{Reverse}^+$}
\SetKwFunction{FEnOPODS}{EnOPODS}
\SetKwProg{Fn}{Function}{:}{end}
      \Begin{
        \SetAlgoVlined
        \If{there exists a $k$-connected orientation $D$ of $G$}
            {Fix an arbitrary linear order on $V$\;
                \FEnODSPrime{$D$, $\emptyset$}\;
            }
      }
        \bigskip
        \Fn{\FEnODSPrime{$D$, $F$}}
        {\SetAlgoVlined
            \eIf{$F \neq V$}
            { take the smallest $v\in V\setminus F$\;
                \FReverseMinus{$D$, $F$, $v$} \;
                \FReversePlus{$D$, $F$, $v$} \;
                \FEnODSPrime{$D$, $F\cup\{v\}$}\;
            }
            {\FEnOPODS{$D$, $\emptyset$}\;}
        }
\end{algorithm}

We need the following easy result for analyzing the amortized complexity:

\begin{lemma}\label{lem:sizeOalpha}
Let $G=(V,E)$ be a graph and $\alpha$ be a $k$-connected out-degree sequence, then $|\mathcal{O}_{\alpha}(G)|\geq(k-1)n+2$.
\end{lemma}
\begin{proof}
Let $D\in \mathcal{O}_{\alpha}(G)$. We will show that $D$ contains at least $(k-1)n+1$ directed cycles. Since for each directed cycle $C$, the orientation $D^C$ is a different element of $\mathcal{O}_{\alpha}(G)$, we obtain the result.

The \emph{cycle space} of $D$ is the set of vectors in $\mathbb{R}^m$ that are linear combinations of \emph{signed incidence vectors} of cycles of $D$. Given a cycle $C$ with a direction of traversal its signed incidence vector has an entry for every arc of $D$, which is $1$ if $a$ is traversed forward, $-1$ if it is traversed backward and $0$ if $a\notin C$. It is well known, that the cycle space of a strongly connected $D$ can be generated by linear combinations of the vectors associated to its directed cycles, see e.g.~\cite{GLS2003}. Moreover, it can be found in most books on (algebraic) graph theory that the dimension of the cycle space of a weakly connected digraph is $m-n+1$, see e.g.~\cite{GR01,KK19}. In particular, $D$ has at least $m-n+1$ directed cycles. Since $D$ is $k$-connected it has at least $kn$ edges. Therefore $D$ contains at least $(k-1)n+1$ directed cycles.  
\qed
\end{proof}

\begin{theorem}
Let $G$ be a graph and $k\in\mathbb{N}$. Algorithm~\ref{algo:final} enumerates all $k$-connected orientations of $G$ with $O(knm^2)$ time delay. If $k\geq 2$ the amortized time is in $O(m^2)$.
\end{theorem}
\begin{proof}
The correctness and the delay follow directly from~\Cref{thm:enumeratesequences} and~\Cref{thm:enumerateOalpha}. Let us compute the amortized time complexity as an average over the delays. Let $s$ be the number of solutions, i.e., the total number of $k$-connected orientations and $t$ be the number of $k$-connected outdegree sequences of $G$. Since by Lemma~\ref{lem:sizeOalpha} for every $k$-connected outdegree sequence $\alpha$ there are at least $(k-1)n+2$ orientations, we have that $t\leq \frac{s}{(k-1)n+2}$. Thus there exist constants $c$ and $c'$, such that the overall runtime of our algorithm is bounded by $cknm^2t+c'm^2s\leq cknm^2\frac{s}{(k-1)n}+c'm^2s=O(m^2)s$, where for the last equality we use $k\geq 2$. Hence the amortized complexity is in $O(m^2)$.
\qed
\end{proof}

\section{Discussion}\label{sec:discuss}
We have given a simple algorithm for enumerating the $k$-arc-connected orientations of a graph with a low amortized time complexity. While polynomial delay algorithms were accessible through the theory of submodular flows before, our result is a contribution to the general quest in enumeration to lower polynomial delays and amortized times as much as possible. Other instances of this quest are the different approaches to enumerate acyclic orientations mentioned in the introduction~\cite{Squire98,Barbosa,Conte2018} and the major open problem of constant amortized time enumeration of the elements of a distributive lattice~\cite{Pruesse1993}. And by~\cite{F04} our enumeration of $\alpha$-orientations can be seen as related to the latter problem.

The weakness of our enumeration algorithm for $k$-arc-connected orientations is that finding the initial solution, i.e., finding a $k$-connected orientation of a graph, is  algorithmically the most complicated part. The best implementations using splitting off techniques are rather simple and lead to computation time of roughly $O(n^5)$~\cite{Gabow94,LauY13}, it would be interesting to find simpler methods.

Our original motivation was the hunt for counterexamples to a conjecture of Neumann-Lara~\cite{NeumannLara85} stating that the vertex set of every oriented planar graph (that is having no directed cycles of length 2) can be vertex-partitioned into two subsets each inducing an acyclic digraph. It is not hard to see that a minimal counterexample to this conjecture has to be a $2$-arc-connected (planar) digraph. However, with a little more work one can actually prove that a minimal counterexample has to be \emph{$2$-vertex-connected}, i.e., at least two vertices have to be removed in order to destroy strong connectivity. For illustration, the left and middle orientations in Figure~\ref{fig:orientations} have vertex-connectivity $1$ while the one on the right has vertex-connectivity $2$. Indeed, using SageMath~\cite{sagemath} one can compute that among the 830918 strong orientations of that graph, only 3842 are $2$-arc-connected, and 3734 of those are $2$-vertex-connected. This also gives an idea, how inefficient it would be to generate all strong orientations and filter with respect to $2$-arc-connectivity.

Deciding if a graph admits a $k$-vertex-connected orientation can be done in polynomial time if and only if $k\leq 2$ (unless P=NP), see~\cite{Thomassen2015,DurandDeGevigney2012}. Thus, this enumeration problem would be viable only for $k=2$. As shown by the middle and right orientation in Figure~\ref{fig:orientations} having the same $\alpha$ does not guarantee the same vertex-connectivity. Thus, our approach does not carry through in this setting. In~\cite{Bang-Jensen2018} it is posed as open problem whether it can be decided in polynomial time, if a mixed graph can be oriented such that it is $2$-vertex-connected. A positive answer would allow polynomial time enumeration as in Algorithm~\ref{algo:submod}. Note that concerning our initial motivation it would be interesting to answer this question even restricted to planar graphs.

A dual analogue of $k$-arc-connectivity could be called \emph{$k$-acyclicity}, where at least $k$ arcs of $D$ have to be contracted in order to destroy its acyclicity. As mentioned in the introduction, acyclic orientations can be enumerated with polynomial delay. On the other hand, it is easy to see that a graph $G$ admits a $2$-acyclic orientation if and only if $G$ is the cover graph of a poset. The corresponding recognition problem is NP-complete~\cite{Brightwell1993} and the proof can be extended to see that testing whether $G$ admits a $k$-acyclic orientation is NP-complete for any $k\geq 2$. 
Using~\cite[Theorem 13]{Cre-16} one can show that enumerating $k$-acyclic orientations (as well as $k$-vertex-connected orientations for $k>2$) are DelNP-hard under D-reductions and IncNP-hard under I-reduction.
Since with an NP-oracle one can decide if a given partial orientation extends to one of these types, a backtrack search algorithm like Algorithm~\ref{algo:submod} yields that the above problems can be solved  with delay a polynomial number of calls of an NP-oracle. Thus these problems are in the class DelNP (and a fortiori IncNP), and therefore  are indeed complete in both settings.

A way of relaxing acyclicity is to consider orientations of \emph{digirth} at least $d$, i.e., all directed cycles are of length at least $d$. We believe that in this setting polynomial delay enumeration should be possible, maybe even when combined with an arc-connectivity requirement.

\subsubsection*{Acknowledgements} We wish to thank Nadia Creignou and Fr\'ed\'eric Olive for fruitful discussions in an early stage of this paper. A preliminary version of the results obtained in this work was presented at \emph{WEPA 2018: Second Workshop on Enumeration Problems and Applications} which was held in Pisa on November 2018. The authors wish to thank the organizers of this workshop.

\bibliographystyle{plain}
\bibliography{biblio.bib}

\end{document}